\documentclass[11pt]{article}

\usepackage{amsfonts,amsmath,amssymb}

\newtheorem{Lem}{Lemma}[section]
\newtheorem{Prop}[Lem]{Proposition}
\newtheorem{Cor}[Lem]{Corollary}
\newtheorem{Thm}[Lem]{Theorem}

\newtheorem{Def}[Lem]{Definition}
\newtheorem{Rem}[Lem]{Remark}

\newtheorem{Expl}[Lem]{Example}
\newtheorem{Expls}[Lem]{Examples}
\newtheorem{Const}[Lem]{Construction}
\newtheorem{Que}[Lem]{Question}
\newenvironment{proof}[1][Proof]{\textbf{#1.} }{$\Box$\medskip\medskip}

\def\NZQ{\Bbb}               

\def\ZZ{{\NZQ Z}}

\newcommand\Tor{\operatorname{Tor}}

\newcommand\Mon{{\operatorname{Mon}}}
\newcommand\Shad{{\operatorname{Shad}}}
\newcommand\indeg{{\operatorname{indeg}}}
\newcommand\supp{{\operatorname{supp}}}
\newcommand\be{\bold{e}}

\newcommand\m{{\operatorname{m}}}

\begin{document}

\title
{Bounding Betti numbers of monomial ideals \\ in the exterior algebra$^1$}

\author{Marilena Crupi, Carmela Ferr\`{o}\\
\ \\
{\footnotesize Department of Mathematics and Computer Science, University of Messina}\\
{\footnotesize Viale Ferdinando Stagno d'Alcontres 31, 98166 Messina, Italy}\\
{\footnotesize e-mail: mcrupi@unime.it; cferro@unime.it}
}  
\date{}
\maketitle

\begin{abstract} Let $K$ be a field, $V$ a $K$-vector space with basis $e_1,\ldots,e_n$, and $E$ the exterior algebra of $V$. To a given monomial ideal $I\subsetneq E$ we associate a special monomial ideal $J$ with generators in the same degrees as those of $I$ and such that the number of the minimal monomial generators in each degree of $I$ and $J$ coincide. We call $J$  the colexsegment ideal associated to $I$. We prove that the class of strongly stable ideals in $E$ generated in one degree satisfies the colex lower bound, that is, the total Betti numbers of
the colexsegment ideal associated to a strongly stable ideal $I\subsetneq E$ generated in one degree are smaller than or equal to those of $I$.



\end{abstract}
\thanks{$^1$ To appear in \textbf{Pure and Applied Mathematics Quarterly}}

\section{Introduction}
Let $K$ be a field. Finding bounds on Betti numbers of classes of graded modules on graded $K$-algebras is one of the main problems in combinatorial algebra.
There are many techniques for finding upper bounds on Betti numbers \cite{AHH, AHH3, AB, JT, HH, CU, CU, MM,KP, Va}. 

On the contrary, the study of lower bounds for Betti numbers is overall harder.
The Buchsbaum-Eisenbud-Horrocks rank conjecture proposes lower bounds for the Betti numbers of a graded module $M$ based on the codimension of $M$. The conjecture was formulated by Buchsbaum and Eisenbud \cite{BE}; independently, the conjecture is implicit in a question by Horrocks \cite{RH}. The literature on special cases of this conjecture is extensive (for example, see \cite{BT, SC1, SC2, HC, DE, EG, TR})  and many recent results have
been obtained in the framework of Boij-S\"{o}derberg theory \cite{BS}.

In this paper, we are interested in establishing lower bounds for Betti numbers of  classes of graded ideals. An important contribution in understanding how to obtain lower bounds for Betti numbers of monomial ideals in a polynomial ring $K[x_1, \ldots, x_n]$ has been done by Nagel and Reiner \cite{NR}. In fact, in \cite{NR}, the authors associate to a given monomial ideal $I$  generated in one degree   a squarefree ideal $J$ generated in the same degree as $I$ by a revlex segment set of squarefree monomials whose length is equal to the number of generators as $I$.   $J$ is called the \textit{colexsegment-generated ideal}. Nagel and Reiner asked the following question.
\begin{Que}
Let $I$ be a monomial ideal in $K[x_1, \ldots, x_n]$ generated in degree $d$, and $J$ the colexsegment-generated ideal. When  does $\beta_i(J) \leq \beta_i(I)$, for all $i\geq 0$, occur?
\end{Que}
If the inequality holds, one says that $I$  \textit{obeys the colex lower bound}.

Their idea \cite{NR} was to give lower bounds for Betti numbers of a monomial ideal not fixing the Hilbert function \cite{TD, CF}, but only fixing the number of minimal monomial generators and their degrees. They proved, among other things, that a strongly stable ideal generated in one degree obeys the colex lower bound. Other classes of monomial ideals in a polynomial ring satisfying such bound where found in \cite{Go}.

The construction in \cite{NR} was generalized in \cite{JB}. The author considered a monomial ideal $I\subsetneq K[x_1, \ldots, x_n]$, not necessarily generated in one degree, and an associated suitable monomial ideal $J$, called the \textit{revlex ideal associated} to $I$. More precisely, if  $p_t$ is the number of minimal generators of $I$ in degree $t$, the minimal generators of $J$ in degree $t$ are the $p_t$ largest monomials in the revlex order not in $\{x_1, \ldots, x_n\}\Mon(J_{t-1})$, where $\Mon(J_{t-1})$ is the set of all monomials of degree $(t-1)$ belonging to $J$.  Since it is possible that the ring $K[x_1, \ldots, x_n]$ has not  enough monomials in some degree, in order to choose the minimal monomial generators for $J$, the author gets around this difficulty by adding extra variables. Thus, he compared the total Betti numbers of  a strongly stable ideal with the total Betti numbers of its revlex ideal associated.

We extend such results in the exterior algebra $E$ of a finitely generated $K$-vector space. Our main result states that a strongly stable ideal in $E$ generated in one degree satisfies the colex lower bound.

The plan of the paper is the following.

Section \ref{intr}  contains preliminary notions and results.

In Section \ref{colex}, if $I\subsetneq E$ is a monomial ideal with generators in several degrees we associate to $I$ a special strongly stable ideal in the exterior algebra $\widetilde E=K\left\langle e_1,\ldots, e_m\right\rangle$, with $m$ sufficiently large, with generators in the same degrees as those of $I$: \textit{the colexsegment ideal associated} to $I$ (Construction \ref{Costr}).

In Section \ref{hil}, we prove some properties satisfied by a strongly stable set in $E$. Therefore, we state an analogue, for the exterior algebra, of Theorem 1.5  in \cite{JB}.

In Section \ref{Betti}, we prove that a strongly stable ideal in $E$ generated in one degree obeys the colex lower bound (Theorem \ref{graded}). Moreover, we discuss the case when a strongly stable ideal in $E$ is generated in several degrees: the colexsegment ideal associated does not give in general a lower bound (Example \ref{case5}).

In Section \ref{rev}, we point out that the colexsegment ideal associated to a monomial ideal is not in general a revlex ideal in the sense of \cite{CF}; then, we determine the conditions allowing the colexsegment ideal associated to a monomial ideal to be a revlex ideal (Proposition \ref{cond}).

\section{Preliminaries and notations} \label{intr}
Let $K$ be a field. We denote by $E=K\left\langle e_1,\ldots , e_n\right\rangle$ the exterior algebra
of a $K$-vector space $V$ with basis $e_1,\ldots,e_n$.
For any subset $\sigma=\{i_1,\ldots,i_d\}$ of $\{1,\dots,n\}$ with $i_1<i_2< \cdots < i_d$ we write
 $e_{\sigma}=e_{i_1} \wedge \ldots \wedge e_{i_d}$, and call $e_{\sigma}$ a monomial of degree $d$. We set $e_{\sigma}=1$, if $\sigma = \emptyset$. The set of monomials in $E$ forms a $K$-basis of $E$ of cardinality $2^n$.

In order to simplify the notation, we put $fg=f \wedge g$ for any two elements $f$ and $g$ in $E.$ An element $f \in E$ is called \textit{homogeneous} of degree $j$ if $f \in E_j,$ where $E_j=\bigwedge^j V$. An ideal $I$ is called \textit{graded} if $I$ is generated by
homogeneous elements. If $I$ is graded, then $I=\oplus_{j \geq 0}I_j$, where $I_j$ is the $K$-vector space of all homogeneous elements $f \in I$ of degree $j$.
We denote by $\indeg(I)$ the \textit{initial degree} of $I$, that is, the minimum $s$ such that $I_s\neq 0.$

If $I$ is a graded ideal in $E$, then $E/I$ has a minimal graded free resolution over $E$:
$$F: \ldots \to F_2 \stackrel{\,\,d_2 \,\,}{\rightarrow} F_1 \stackrel{\,\, d_1 \,\,}{\rightarrow} F_0 \to E/I \to 0,$$
where $F_i=\oplus_jE(-j)^{\beta_{i,j}(E/I)}$. The integers $\beta_{i,j}(E/I)=\dim_K\Tor_{i}^{E}(E/I,K)_j$
are called the \textit{graded Betti numbers} of $E/I$,
whereas the numbers $\beta_{i}(E/I)=\sum_j\beta_{i,j}(E/I)$ are called the \textit{total Betti numbers} of $E/I.$

Let $u$ be a monomial in $E$. We define
\[
\supp(u)=\{\mbox{$i$\,:\,$e_i$ divides $u$}\},
\]
and we write
\[
\m(u)=\max\{i: i \in \supp(u)\}, \qquad \min(u)=\min\{i: i \in \supp(u)\}.
\]

For any subset $S$ of $E$, we denote by $\Mon(S)$ the set of all monomials in $S$, by $\Mon_d(S)$ the set of all monomials  of degree $d \geq 1$ in $S$, and by $\vert S \vert$ its cardinality.

\begin{Def} A nonempty set $M \subseteq \Mon_d(E)$ is called \textit{stable} if for each monomial $e_{\sigma} \in M$ and each $j < \m(e_{\sigma})$ one has
$(-1)^{\alpha(\sigma,j)}e_j e_{{\sigma} \setminus \{\m(e_{\sigma})\}} \in M$, where $\alpha(\sigma,j)=\vert\{r \in \sigma : r < j\}\vert.$
$M$ is called \textit{strongly stable} if for all $e_{\sigma} \in M$ and all $j \in \sigma$ one has $(-1)^{\alpha(\sigma,i)}e_ie_{\sigma \setminus \{j\}} \in M$, for all $i<j$, where $\alpha(\sigma,i)=\vert\{r \in \sigma : r <i\}\vert.$
\end{Def}
\begin{Def} Let $I$ be a monomial ideal of $E$. $I$ is called \textit{stable} if for each monomial $e_{\sigma}\in I$ and each $j < \m(e_{\sigma})$ one has $e_j e_{{\sigma} \setminus \{\m(e_{\sigma})\}} \in I$.

$I$ is called \textit{strongly stable} if for each monomial $e_{\sigma} \in I$ and each $j \in \sigma$ one has $e_ie_{\sigma \setminus \{j\}} \in I$, for all $i<j$.
\end{Def}
\begin{Rem} Note that a monomial ideal $I$ of $E$ is a (strongly) stable ideal in
$E$ if and only if $\Mon(I_d)$ is a (strongly) stable set in $E$ for all $d$.
\end{Rem}

\begin{Def} \label{DefShadow}
Let $M$ be a set of monomials of $E$. Set $\bold e_i = \{e_1, \ldots, e_i\}.$ We define the set
$$\bold e_i M= \{(-1)^{\alpha(\sigma,j)}e_je_{\sigma} : \mbox{$ \, e_{\sigma} \in M$, \,\, $j \notin \supp(e_{\sigma})$, \,\, $j=1, \ldots, i$}\},$$
$\alpha(\sigma,j)=\vert\{r \in \sigma : r <j\}\vert$.
\end{Def}
Note that $\bold e_i M= \emptyset$ if, for every monomial $u \in M$ and for every $j=1, \ldots, i$, one has  $j\in \supp(u).$

If $M$ is a set of monomials of degree $d<n$ of $E$,  $\bold e_n M$ is called the \textit{shadow} of $M$ and is denoted by $\Shad(M)$:
\[\Shad(M) = \{(-1)^{\alpha(\sigma,j)}e_je_{\sigma} : \, e_{\sigma} \in M,\,\, j \notin \supp(e_{\sigma}), \,\, \mbox{ $j = 1, \ldots, n$}\},\]
$\alpha(\sigma,j)=\vert\{r \in \sigma : r <j\}\vert$. Moreover, we denote by $E_1M$ the $K$-vector space generated by $\Shad(M)$.
\begin{Rem} Usually, the shadow of a set $M$ of monomials of degree $d$  of $E$, $d<n$, is defined as follows:
\[\Shad(M) = \{e_je_{\sigma} : \, e_{\sigma} \in M,\,\, j \notin \supp(e_{\sigma}), \,\, \mbox{ $j = 1, \ldots, n$}\}.\]
We observe that this definition is a little bit imprecise.

In fact, if $j <\min(e_{\sigma})$, then $e_je_{\sigma}\in\Mon_{d+1}(E)$. Suppose $j > \min(e_{\sigma})$ and $e_{\sigma} = e_{i_1}e_{i_2}\cdots e_{i_d}$.
Since $e_he_i = -e_ie_h$, $i, h\in \{1, \ldots, n\}$, then $e_je_{\sigma}= (-1)^t e_{i_1}e_{i_2}\cdots e_{i_t} e_j e_{i_{t+1}} \cdots$ $e_{i_d}$, where $t$ is the largest integer such that $i_t < j $, that is, $t = \alpha(\sigma,j)$. Note that if $t$ is odd, then $e_je_{\sigma}\notin \Mon_{d+1}(E)$.
\end{Rem}

Finally, if $I$ is a monomial ideal of $E$, we denote by $G(I)$ the unique minimal set of monomial generators of $I$, and define the following sets:
$$G(I)_d = \{u \in G(I): \,\, \deg(u)=d\},\,\,G(I;i)=\{u \in G(I): \,\, \m(u)=i\},$$
 $$\m_i(I)=\left| G(I;i) \right| , \,\,\, \m_{\leq i}(I)= \sum_{j \leq i}\m_j(I),$$
 for $d>0$ and $1 \leq i \leq n$.

\section{Colexsegment ideal associated to a monomial ideal}\label{colex}

In this Section,
to a given monomial ideal $I\subsetneq E$ we associate a special monomial ideal $J$ with generators in the same degrees as those of $I$ and such that the number of the minimal monomial generators in each degree of $I$ and $J$ coincide.

Let us denote by $>_{\textrm{revlex}}$
the \textit{reverse lexicographic order }(revlex order, for short) on $\Mon_d(E)$, that is, if $e_{\sigma}=e_{i_1}e_{i_2}\cdots e_{i_d}$ and
 $e_{\tau}=e_{j_1}e_{j_2}\cdots e_{j_d}$ are monomials belonging to $\Mon_d(E)$ with  $1 \leq i_1<i_2< \cdots < i_d \leq n$ and
 $1 \leq j_1<j_2< \cdots < j_d \leq n$, then
 \[e_{\sigma} >_{\textrm{revlex}} e_{\tau} \,\,\, \text{if} \,\,\, i_d=j_d, i_{d-1}=j_{d-1}, \ldots , i_{s+1}=j_{s+1} \,\, \text{and} \,\, i_s<j_s
 \,\, \text{for some} \,\, 1 \leq s \leq d.\]
\begin{Def}
 A nonempty set $M \subseteq \Mon_d(E)$ is called a \textit{reverse lexicographic segment of degree d }(revlex segment of degree $d$, for short) if
  for all $v \in M$ and all $ u \in \Mon_d(E)$ such that $u >_{\textrm{revlex}} v,$ we have that $u \in M$.
\end{Def}
If $M$ is a revlex segment of degree $d$ and $\vert M \vert =\ell$, $\ell$ is called the \textit{length} of $M$.

Following \cite{NR} and \cite{JB}, we give the following construction.
\begin{Const} \label{Costr}
Let $I\subsetneq E=K \left\langle e_1, \ldots, e_n\right\rangle$ be a monomial ideal generated in degrees $1\leq d_1< d_2 < \ldots < d_t\leq n$. Let $p_j$
be the number of minimal generators of $I$ in degree $d_j$. We construct a monomial ideal $J$ in the exterior algebra $\widetilde E = K\left\langle e_1, \ldots, e_m\right\rangle$ by choosing the minimal generators
as follows:\\
for each $d_j$ the degree $d_j$ minimal generators of $J$ are the largest $p_j$ monomials in the revlex order not in $\widetilde E_1J_{d_{j-1}}$.\\
The integer $m$ is the smallest integer such that there exists an exterior algebra $\widetilde E = K\left\langle e_1,\ldots, e_m\right\rangle$ with enough monomials such that the construction can be completed.
\end{Const}

$J$ is called the \textit{colexsegment ideal associated} to $I$.
\begin{Rem} Set $[n] = \{1, \ldots, n\}$. Let $d\in [n]$ and $\binom {[n]} d$ the set of all $d$-subsets $S_d = \{i_1, i_2, \ldots, i_d\,:\, 1 \leq i_1 < i_2 < \cdots < i_d \leq n\}$ of $[n]$.
Using the revlex order on $\Mon_d(E)$, we can arrange the set $\binom {[n]} d$ by the colexicographic order
\cite{NR} as follows.

If $S_d = \{i_1, i_2, \ldots, i_d\,:\, 1\leq i_1 < i_2 < \cdots < i_d\leq n\}$ and $T_d = \{j_1, j_2, \ldots, j_d\,:\, 1\leq j_1 < j_2 < \cdots < j_d\leq n\}$ are two elements of $\binom {[n]} d$, then
\[\mbox{$S_d <_{\textrm{colex}} T_d$\qquad if \qquad $e_{i_1}e_{i_2}\cdots e_{i_d}>_{\textrm{revlex}}e_{j_1}e_{j_2}\cdots e_{j_d}$}. \]

Hence, if the ideal $I$ in  Construction \ref{Costr} is generated in degree $d$, then $J$ is the analogous of the \textit{colexsegment-generated ideal} of \cite{NR}.
\end{Rem}

\begin{Rem} If $M$ is a revlex segment of degree $d$ and length $\ell$ in the exterior algebra $E=K \left\langle e_1,\ldots, e_n\right\rangle$, then $M$ is also a revlex segment of degree $d$ and length $\ell$ in the exterior algebra $\widetilde E = K\left\langle e_1,\ldots, e_m\right\rangle$, for $m\geq n$.

Moreover, if $I\subsetneq E$ is a monomial ideal generated in degree $d$, then the colexsegment ideal $J$ associated to $I$ is an ideal of $E$, too.

Note that if $I\subsetneq E$ is a monomial ideal, then the generators of the colexsegment ideals associated to $I$ are the same for every construction done with some $\widetilde m \geq m$, where $m$ is the integer defined in Construction \ref{Costr}.
\end{Rem}
Below we give two examples to show that, given a monomial ideal $I \subsetneq E = K\left\langle e_1, \ldots , e_n\right\rangle$, it is not always verified that $E$ has enough monomials in some degree, in order to choose the minimal monomial generators for the colexsegment ideal $J$ associated to $I$.

\begin{Expl} \label{3.2}
Let $I=(e_1e_2, e_1e_3e_4, e_1e_3e_5)\subsetneq E = K\left\langle e_1, e_2, e_3, e_4, e_5\right\rangle$ be a monomial ideal generated in degrees $2$ and $3$. Following Construction \ref{Costr}, the colexsegment ideal associated to $I$ is $J=(e_1e_2, e_1e_3e_4, e_2e_3e_4)\subsetneq E$.
\end{Expl}
\begin{Expl}\label{3.3}
Let $I=(e_1e_2, e_1e_3, e_1e_4, e_1e_5, e_2e_3e_4, e_2e_3e_5, e_2e_4e_5)$ be a monomial ideal of $E= K\left\langle e_1, e_2, e_3, e_4, e_5\right\rangle$ generated in degrees $2$ and $3$. The colexsegment ideal associated to $I$ is the ideal $J=(e_1e_2, e_1e_3, e_2e_3, e_1e_4, e_2e_4e_5, e_3e_4e_5, e_2e_4e_6)$ in the exterior algebra $\widetilde E=K\left\langle e_1, e_2, e_3, e_4, e_5, e_6\right\rangle$.\\
We do not have enough monomials of degree $3$ in $E$ to get Construction \ref{Costr}.
\end{Expl}

For the remainder of this paper we will assume that the dimension of the $K$-vector space $V$ on which we construct the exterior algebra $E$ is
sufficiently large to construct the colexsegment ideal.
\section{Green's Theorem for colexsegment ideals in the exterior algebra}\label{hil}
In this Section, we extend Theorem 1.5 in \cite{JB} to the case of the exterior algebra. Biermann's Theorem is an analogue of Green's Theorem \cite{MG} for colexsegment ideals in polynomial rings.

For any subset $M \subseteq \Mon_d(E)$ we denote by $\min(M)$ the smallest monomial belonging to $M$ with respect to the revlex order on $E$.

From now on, in order to shorten the notations, we will write $>$ instead of $>_{\textrm{revlex}}.$

Let $M$ be a set of monomials in $E=K \left\langle e_1,\ldots, e_n\right\rangle$. For $1\leq p \leq n$, define
\[M_{\leq p} = \{u\in M\,:\, \m(u)\leq p\}, \qquad \m_{\leq p}(M) = \vert M_{\leq p}\vert.\]

\begin{Lem} \label{Bi} Let $M$ be a strongly stable set in $E$. Set $e_{\sigma}:=\min(M)$. Then
$e_ie_{\sigma} \in \Shad(M\setminus \{e_{\sigma}\})$  if and only if for $i \notin\supp(e_{\sigma})$ one has $i < \m(e_{\sigma})$.

\end{Lem}

\begin{proof} 
If $e_ie_{\sigma} \in \Shad(M\setminus \{e_{\sigma}\})$, $i \notin\supp(e_{\sigma})$, then $e_ie_{\sigma} = e_je_{\tau}$, where $e_{\tau} \in M\setminus \{e_{\sigma}\}$. By the meaning of  $e_{\sigma}$, it follows that $e_{\tau} > e_{\sigma}$ and $i < j$. On the other hand, $e_j$ divides $e_{\sigma}$ and so $ j \leq \m(e_{\sigma})$. Hence, $i < \m(e_{\sigma})$.

Conversely, let $i < \m(e_{\sigma})$, $i \notin\supp(e_{\sigma})$.
Since $M$ is a strongly stable set, then $e_{\tau}=(-1)^{\alpha(\sigma,i)}e_ie_{{\sigma} \setminus \{\m(e_{\sigma})\}}\in M\setminus \{e_{\sigma}\}$. Therefore, $(-1)^{\alpha(\sigma,i)}e_ie_{\sigma} = e_{\tau} e_{\m(e_{\sigma})}\in \Shad(M\setminus \{e_{\sigma}\})$ follows.
\end{proof}

\begin{Prop} \label{Hsquare} Let $M$ be a strongly stable set in $E$. Then
\[\Shad(M) = \bigcup_{e_{\sigma}\in M}\{e_{\sigma}e_{\m(e_{\sigma})+1}, \ldots, e_{\sigma}e_n\},\]
where $\bigcup_{e_{\sigma}\in M}$ is a disjoint union.
\end{Prop}
\begin{proof} We proceed by induction on $\vert M\vert$.
Let $\vert M\vert = 1$. Since $M = \{e_1e_2\cdots e_d\}$, then
\[\Shad(M) = \{e_1e_2\cdots e_de_r\,:\, r = d+1, \ldots, n\}\]
and the assertion follows.

Let $\vert M\vert > 1$, and assume the assertion to be true if the cardinality is smaller than $\vert M\vert$.
Set $e_{\tau}:=\min(M)$. Therefore, $\Shad(M) = \Shad(\{e_{\tau}\}) \bigcup \Shad(M \setminus \{e_{\tau}\})$.
From Lemma \ref{Bi}, we have that $e_ie_{\tau} \in \Shad(M\setminus \{e_{\tau}\})$ for every $i < \m(e_{\tau})$, $i \notin\supp(e_{\tau})$. Hence,
\[\Shad(M) =\{e_{\tau}e_{\m(e_{\tau})+1}, \ldots, e_{\tau}e_n\} \bigcup \Shad(M \setminus \{e_{\tau}\}).\]

Since $\Shad(M \setminus \{e_{\tau}\})$ is a strongly stable set \cite{JT}, by the inductive hypothesis, we have that
\[\Shad(M \setminus \{e_{\tau}\}) = \bigcup_{e_{\sigma}\in M\setminus \{e_{\tau}\}}\{e_{\sigma}e_{\m(e_{\sigma})+1}, \ldots, e_{\sigma}e_n\},\]
and, as a consequence, $\Shad(M) = \bigcup_{e_{\sigma}\in M}\{e_{\sigma}e_{\m(e_{\sigma})+1}, \ldots, e_{\sigma}e_n\}$.

Now we prove that the union above is disjoint.
From Lemma \ref{Bi}, we have that, if $i>\m(e_{\tau})$, then $e_ie_{\tau} \notin \Shad(M\setminus \{e_{\tau}\})$, whereupon
\[\{e_{\tau}e_{\m(e_{\tau})+1}, \ldots, e_{\tau}e_n\}\bigcap \Shad(M\setminus \{e_{\tau}\}) = \emptyset.\]

It follows that, for every $e_{\sigma} > e_{\tau}$, we have
$\{e_{\tau}e_{\m(e_{\tau})+1}, \ldots, e_{\tau}e_n\}\bigcap$ $\{e_{\sigma}e_{\m(e_{\sigma})+1}, \ldots, e_{\sigma}e_n\}$ $= \emptyset$.
\end{proof}

As a consequence, the following corollary is obtained.

\begin{Cor} \label{sha} Let $I$ be a strongly stable ideal in $E$. Then for all integers $p$ such that $t+1\leq p \leq n$ with $t\geq \indeg(I)$, we have
\[\be_p\Mon(I_t)_{\leq p} = \bigcup_{i=t+1}^p e_i\Mon(I_t)_{\leq i}.\]
\end{Cor}
\begin{proof}
It is sufficient to observe that $\be_p\Mon(I_t)_{\leq p} = \emptyset$, for $p \leq t$.
\end{proof}

\begin{Rem} The previous results are the generalization in the exterior algebra of results due to Bigatti \cite{AB} in the polynomial case (see also \cite{AHH2}).
\end{Rem}
\begin{Thm} \label{green} Let $I$ be a strongly stable ideal in $E$ and $J$ the colexsegment associated to $I$. Then for all integers $p$ such that $t\leq p \leq n$ with $t\geq \indeg(I)$, we have
\[\m_{\leq p}(\Mon(I_t)) \leq \m_{\leq p}(\Mon(J_t)).\]
\end{Thm}

\begin{proof}
We proceed by induction on $t\geq \indeg(I)$.

Let $d = \indeg (I) = \indeg (J)$. We have that the sets $\Mon(I_d)_{\leq p}$ and $\Mon(J_d)_{\leq p}$ consist only of minimal generators of $I$ and $J$. If $u$ and $v$ are monomials of the same degree and $\m(u) < \m(v)$, then $u>v$ in the revlex order. By construction,
the minimal generators of $J$ in degree $d$ form a revlex segment of degree $d$. Hence, since $\vert G(I)_d \vert = \vert G(J)_d\vert$, for all $p\in \{d, \ldots, n\}$, we have the inequalities
\[\m_{\leq p}(\Mon(I_d)) \leq \m_{\leq p}(\Mon(J_d)).\]

Now suppose that $\m_{\leq p}(\Mon(I_{t-1})) \leq \m_{\leq p}(\Mon(J_{t-1}))$, for all $p\in \{t-1,\ldots, n\}$.

Let $p\in \{t,\ldots, n\}$. The set of $\Mon(I_t)_{\leq p}$ contains two kinds of monomials: minimal generators of $I$ in degree $t$ and monomials belonging to $\be_p \Mon(I_{t-1})_{\leq p}$.
From Corollary \ref{sha} and by the inductive hypothesis, we have
\begin{eqnarray}\label{1}
  \vert \be_p \Mon(I_{t-1})_{\leq p} \vert & = & \vert \bigcup_{i=t}^p e_i \Mon(I_{t-1})_{\leq i} \vert  = \sum_{i=t}^p \m_{\leq i}(\Mon(I_{t-1}))\leq \\
  \nonumber  \leq  \sum_{i=t}^p\m_{\leq i}(\Mon(J_{t-1})) &=& \vert \bigcup_{i=t}^p e_i \Mon(J_{t-1})_{\leq i} \vert =  \vert \be_p \Mon(J_{t-1})_{\leq p} \vert.
\end{eqnarray}

Now we focus our attention on $G(I)_t$ and $G(J)_t$.
We distinguish two cases.\\
(Case 1.) Let $\Mon(J_t)_{\leq p} = \Mon_t(E)_{\leq p}$. In this situation, it is easy to verify that $\m_{\leq p}(\Mon(I_t)) \leq \m_{\leq p}(\Mon(J_t))$, for all $p\in \{t,\ldots, n\}$.\\
(Case 2.) Let $\Mon(J_t)_{\leq p} \subsetneq \Mon_t(E)_{\leq p}$. From the behaviour of the revlex order it follows that all the degree $t$ generators of $J$ are in the set $\Mon(J_t)_{\leq p}$.  On the other hand, $\vert G(I)_t \vert = \vert G(J)_t \vert$ and by (\ref{1}), we get the desired inequalities.
\end{proof}

\begin{Rem} The previous results point out that if $R$ is a revlex segment of degree $d$ in an exterior algebra $E$ and $X$ is a strongly stable set of degree $d$ in $E$ with the same cardinality as $R$, then $\vert \Shad(R)\vert  \geq \vert \Shad(X)\vert$.
\end{Rem}

\section{The colex lower bound in the exterior algebra}\label{Betti}
Let $I\subsetneq E$ be a monomial ideal generated in degree $d$, and $J$ the colexsegment ideal associated to $I$.
We say that $I$ \textit{satisfies the colex lower bound} if $\beta_i(J) \leq \beta_i(I)$, for all $i\geq 0$.
In this Section we prove that the class of strongly stable ideals generated in one degree obeys the colex lower bound.

If $I\subsetneq E$ is a strongly stable ideal generated in several degrees, then the colexsegment ideal does not give in general a lower bound, as the following example shows.

\begin{Expl} \label{case5} Let $E=K\left\langle e_1,e_2,e_3,e_4,e_5\right\rangle$. Denote by $I$ a strongly stable ideal in $E$ generated in several degrees and by $J$ the colexsegment ideal associated to $I$.

Let  $I$ and $J$ be the ideals described in the following table:
\begin{center}
\begin{tabular}{|l|l|}
\hline
Strongly stable ideal & Colexsegment ideal associated \\
    \hline
    $I =(e_1e_2,e_1e_3e_4,e_1e_3e_5)$ & $J =(e_1e_2,e_1e_3e_4,e_2e_3e_4)$ \\
    $I= (e_1e_2,e_1e_3e_4,e_1e_3e_5,e_1e_4e_5)$ & $J = (e_1e_2,e_1e_3e_4,e_2e_3e_4,e_1e_3e_5)$ \\
    $I=(e_1e_2,e_1e_3,e_1e_4e_5)$ & $J=(e_1e_2,e_1e_3,e_2e_3e_4)$ \\
    $I=(e_1e_2,e_1e_3,e_1e_4,e_1e_5,e_2e_3e_4)$ & $J=(e_1e_2,e_1e_3,e_2e_3,e_1e_4,e_2e_4e_5)$ \\
    $I= (e_1e_2,e_1e_3,e_1e_4,e_1e_5,e_2e_3e_4,e_2e_3e_5)$ & $J=(e_1e_2,e_1e_3,e_2e_3,e_1e_4,e_2e_4e_5,e_3e_4e_5)$ \\
    $I= (e_1e_2,e_1e_3,e_1e_4,e_1e_5,e_2e_3,e_2e_4e_5)$ & $J=(e_1e_2,e_1e_3,e_2e_3,e_1e_4,e_2e_4,e_3e_4e_5)$ \\
    $I=(e_1e_2,e_1e_3,e_1e_4,e_1e_5,e_2e_3e_4,e_2e_3e_5)$ & $J=(e_1e_2,e_1e_3,e_2e_3,e_1e_4,e_2e_4e_5,e_3e_4e_5)$ \\
    $I=(e_1e_2e_3,e_1e_2e_4,e_1e_2e_5,e_1e_3e_4e_5)$ & $J=(e_1e_2e_3,e_1e_2e_4,e_1e_3e_4,e_2e_3e_4e_5)$ \\
    \hline
  \end{tabular}
\end{center}
\noindent
then, $\beta_i(J) \leq \beta_i(I)$, for all $i\geq 0$ \cite[Corollary 3.3]{AHH}. In all these cases the colexsegment ideal $J$ gives a \textit{lower bound} for the total Betti numbers of $I$.

On the contrary, let $I$ and $J$ be the ideals described in the following table:
\begin{center}
\begin{tabular}{|l|l|}
\hline
Strongly stable ideal & Colexsegment ideal associated \\
    \hline
    $I =(e_1e_2,e_1e_3,e_1e_4,e_2e_3e_4)$ & $J =(e_1e_2,e_1e_3,e_2e_3,e_1e_4e_5)$ \\
    $I= (e_1e_2,e_1e_3,e_1e_4,e_2e_3e_4,e_2e_3e_5)$ & $J = (e_1e_2,e_1e_3,e_2e_3,e_1e_4e_5,e_2e_4e_5)$ \\
    $I=(e_1e_2,e_1e_3,e_1e_4,e_2e_3e_4,e_2e_3e_5,e_2e_4e_5)$ & $J=(e_1e_2,e_1e_3,e_2e_3,e_1e_4e_5,e_2e_4e_5,e_3e_4e_5)$ \\
\hline
  \end{tabular}
\end{center}
\noindent
then, $\beta_i(J) > \beta_i(I)$, for all $i\geq 0$ \cite[Corollary 3.3]{AHH}. In all these cases $J$ gives an \textit{upper bound} for the total Betti numbers of $I$.

In all other cases not included in the above tables it is $\beta_i(J) = \beta_i(I)$, for all $i\geq 0.$
\end{Expl}

For a strongly stable ideal generated in one degree we state the following result.

\begin{Thm}\label{graded}
Let $I\subsetneq E$ be a strongly stable ideal generated in degree $d$ and $J$ the colexsegment ideal associated to $I$. Then $I$ satisfies the colex lower bound.
\end{Thm}
\begin{proof}
Let $\{u_1, \ldots, u_r\}$ and $\{v_1, \ldots, v_r\}$ be the minimal systems of monomial generators of $I$ and $J$, respectively. We may assume that these generators are ordered so that $\m(u_i) \leq \m(u_j)$ and $\m(v_i) \leq \m(v_j)$, for all $i<j.$

From \cite[Corollary 3.3]{AHH}, since $I$ and $J$ are strongly stable ideals generated in degree $d$, for every $i\geq 0$, we get:
\begin{align*}
\beta_i(I) = \beta_{i,i+d}(I) &=\sum_{u \in G(I)}\binom{\m(u)+i-1}{\m(u)-1}&\\
&=\sum_{t=d}^{n}\binom{t+i-1}{t-1}\left|u \in G(I) : \m(u)=t \right| &\\
&=\sum_{t=d}^{n}\binom{t+i-1}{t-1} \m_t(I),
\end{align*}
and, similarly,
\[\beta_i(J) = \beta_{i,i+d}(J)=\sum_{t=d}^{n}\binom{t+i-1}{t-1} \m_t(J),\,\, \textrm{for all}\,\, i \geq 0.\]
We will prove that, for all $1 \leq j \leq r$, it is
\begin{equation}\label{dis1}
    \m(v_j)\leq \m(u_j).
\end{equation}

Set $\m(u_r)=k$. From Theorem \ref{green}, $\m_{\leq k}(\Mon(I_t)) \leq \m_{\leq k}(\Mon(J_t))$ holds true, and the assertion follows.
\end{proof}

\section{Colexsegment ideals which are revlex ideals}\label{rev}
In this Section we analyze when a colexsegment ideal is a revlex ideal.

In \cite{CF}, the following definition is given.
\begin{Def}\label{revdef}
Let $I=\oplus_{j \geq 0}I_j$ be a monomial ideal in $E.$ We say that $I$ is a \textit{reverse lexicographic ideal} (revlex ideal, for short) if, for every $j$, $I_j$ is spanned  (as  $K$-vector space) by a revlex segment.
\end{Def}

If $I\subsetneq E$ is a monomial ideal and $J$ is the colexsegment ideal associated to $I$, then Construction \ref{Costr} does not guarantee that $J$ is a revlex ideal. One can only say that $J$ is a strongly stable ideal in $E$. The reason is that the shadow of a revlex segment of degree $d$ needs not to be a revlex segment of degree $d+1$ \cite{CF}.

\begin{Expls} (1).
Let $I=(e_1e_2, e_1e_3, e_1e_4e_5)\subsetneq E = K\left\langle e_1, e_2, e_3, e_4, e_5,e_6\right\rangle.$
The colexsegment ideal associated to $I$ is $J=(e_1e_2, e_1e_3,e_2e_3e_4)\subsetneq E$ which is not a revlex
 ideal.
Indeed, $J_3$ is not generated as a $K$-vector space by a revlex segment of degree $3$. In fact, the monomial $e_1e_2e_6$ belongs to $\Mon(J_3)$, but the monomial $e_3e_4e_5$,
that is greater than $e_1e_2e_6$, does not belong to $\Mon(J_3)$.\\
(2). Let $I=(e_1e_2, e_1e_3, e_1e_4, e_2e_3e_4)\subsetneq E = K\left\langle e_1, e_2, e_3, e_4, e_5\right\rangle$. The colexsegment ideal associated to $I$ is $J=(e_1e_2, e_1e_3,e_2e_3, e_1e_4e_5)\subsetneq E$, which  is a revlex ideal.
\end{Expls}

We quote the next results from \cite{CF, CF2}.
\begin{Prop} \cite[Corollary 3.9]{CF}\label{PropCF}
Let $M = \{e_{\sigma_1}, \ldots, e_{\sigma_t}\}$ be a set of monomials in $E$ and let $d_1 = \min\{\deg(e_{\sigma_i})\,:\, i=1, \ldots, t\}$ and $d_2 = \max\{\deg(e_{\sigma_i})\,:\, i=1, \ldots, t\}$, with $d_2 < n-2$. Then $I = (M)$ is revlex ideal if and only if
\begin{enumerate}
	\item[\em(1)] $I_j$ is a revlex segment for $d_1 \leq j \leq d_2$;
	\item[\em(2)] $e_{n-(d_2+1)} \cdots e_{n-3}e_{n-2} \in M$.
\end{enumerate}
\end{Prop}

\begin{Prop} \cite[Proposition 2.1]{CF2}\label{PropCF2}
Let $M$ be a revlex segment of degree $d< n-2$ in $E$. Then the following conditions are equivalent:
\begin{enumerate}
\item[\em(a)] $\Shad(M)$ is a revlex segment of degree $d+1$;
\item[\em(b)] $\vert M \vert \geq \binom{n-2}{d}$;
\item[\em(c)] $e_{n-(d+1)} \cdots e_{n-2} \in M$;
\end{enumerate}
\end{Prop}

In \cite[Proposition 3.1]{CF2}, we have proved that a revlex ideal in an exterior algebra is minimally generated in at most two consecutive degrees. This fact, together with the conditions forced by Construction \ref{Costr}, justifies our assumptions in the next result.

\begin{Prop} \label{cond}
Let $I\subsetneq E$ be a monomial ideal generated in degrees $d_1< d_2 < n-2$, and $J$ the colexsegment ideal associated to $I$.
$J$ is a revlex ideal in $E$ generated in degrees $d_1< d_2$ if and only if one of the following conditions holds true:

\begin{enumerate}
\item[\em(i)] $\dim_K I_{d_1} \geq \binom {n-2}{d_1}$;
\item[\em(ii)]$d_2 = d_1+1$ and $\dim_K I_{d_2} \geq \sum_{r=d_1}^{n-2}\binom r{d_1} + c$,\\
where $c=\left|\{ v \in M_{d_2} : e_{n-(d_1+1)}\cdots e_{n-2}e_{n-1} > v \geq \min(\Shad(\Mon(I_{d_1})))\}\right|$.
\end{enumerate}
\end{Prop}
\begin{proof}
Let $J$ be the colexsegment ideal associated to $I$. Suppose $J$ is a revlex ideal.\\
(Case 1.) Let $e_{n-(d_1+1)}\cdots e_{n-2}\in G(J)_{d_1}$. Then $\vert G(J)_{d_1}\vert \geq \binom {n-2}{d_1}$ (Proposition \ref{PropCF2}). Since by definition $\dim_K I_{d_1} = \vert G(I)_{d_1}\vert  = \vert G(J)_{d_1}\vert = \dim_K J_{d_1}$, condition (i) follows.\\
(Case 2.) Let $e_{n-(d_1+1)}\cdots e_{n-2}\notin G(J)_{d_1}.$ Since $J$ is a revlex ideal of $ E$ and, consequently, $J_{d_1+1}$ is a revlex segment of degree $d_1+1$, it follows that $e_{n-(d_1+1)}\cdots e_{n-2}e_{n-1}\in \Mon(J_{d_1+1})$.

In fact, $e_1e_2\cdots e_{d_1}e_n\in E_1J_{d_1}\subseteq J_{d_1+1}$ and $e_{n-(d_1+1)}\cdots e_{n-2}e_{n-1} > e_1e_2 \cdots e_{d_1}e_n$ implies
$e_{n-(d_1+1)}\cdots e_{n-2}e_{n-1}\in \Mon(J_{d_1+1})$.
Setting \[w:=\min(\Shad(\Mon(J_{d_1}))), \quad z:=e_{n-(d_1+1)}\cdots e_{n-2}e_{n-1},\]
consider the following sets:
\[\mbox{$A=\{u \in M_{d_2} : u \geq z\}$, \qquad $B=\{v \in M_{d_2} : z > v \geq w\}$}.\]
From Construction \ref{Costr}, $\dim_K I_{d_2} \geq \left|A\right|+ \left|B\right|$. Since  $\left|A\right|=\sum_{r=d_1}^{n-2}\binom{r}{d_1}$, then condition (ii) follows.

Conversely, suppose condition (i) holds. Since $\dim_K I_{d_1} = \dim_K J_{d_1}$, then the assertion follows from Proposition \ref{PropCF}.
In fact, from Proposition \ref{PropCF2}, $\Shad(\Mon(J_{d_1}))$ is a revlex segment of degree $d_1+1$. Hence, $\Mon(J_{d_1+1})$ is a revlex segment of degree $d_1+1$, too.

Suppose condition (ii) holds. Since by construction $\vert G(I)_{d_1+1} \vert = \vert G(J)_{d_1+1} \vert$, therefore
$e_{n-(d_1+2)}\cdots e_{n-2} \in \Mon(J_{d_1+1})$. On the other hand, by construction $\Mon(J_{d_1})$ is a revlex segment.
Hence, $J$ is a revlex ideal by Proposition \ref{PropCF}.
\end{proof}
\begin{Cor} \label{strongly} Let $I\subsetneq E$ be a monomial ideal generated in degree $d<n-2$. Then the colexsegment ideal $J\subsetneq E$ associated to $I$ is a revlex ideal if and only if $\left|G(I)\right| \geq \binom{n-2}{d}$.
\end{Cor}


\begin{thebibliography}{99}

\bibitem{AHH} {A. Aramova, J. Herzog, T. Hibi}, \newblock {\em Gotzman Theorems for Exterior algebra and combinatorics}, \newblock J. Algebra {\bf 191} (1997), 174--211.

\bibitem{AHH2} {A. Aramova, J. Herzog, T. Hibi}, \newblock {\em Squarefree lexsegment ideals}, \newblock Math. Z. 
{\bf 228} (1998), 353--378.

\bibitem{AHH3} {A. Aramova, J. Herzog, T. Hibi}, \newblock {\em Shifting operations and graded Betti numbers},
J. Algebraic Combin. {\bf 12} (2000), 207--222.

\bibitem{JB} {J. Biermann},  \newblock {\em Reverse Lex Ideals}, \newblock J. Algebra {\bf 323} (2010), 749--756.

\bibitem{AB} {A. Bigatti},  \newblock{\em Upper bounds for the Betti numbers of a given Hilbert function}, \newblock
Comm. Algebra {\bf 21} (1993), 2317--2334.

\bibitem{BS} {M. Boij, J. S\"{o}derberg}, \newblock {\em Betti numbers of graded modules and the Multiplicity Conjecture in the non-Cohen-Macaulay case}, J. Algebra Number Theory {\bf 6}(3) (2009), 437--454.

\bibitem{BJV} {H. Brenner, J. Herzog, O. Villamayor}. {\em Three lectures on Commutative algebra}, University Lecture Series {\bf 42}, AMS-RSME, 2008.

\bibitem{BT} {M. Brun, T. R\"{o}mer}, \newblock {\em Betti numbers of $\ZZ^n$-graded modules}, Comm. Algebra {\bf 32}
(2004), 4589--4599.

\bibitem{BE} {D.A. Buchsbaum, D. Eisenbud}, \newblock {\em Algebra structures for finite free resolutions, and some
structure theorems for ideals of codimension 3}, \newblock Amer. J. Math. {\bf 99} (1977), 447--485.

\bibitem{SC1} {S.T Chang}, \newblock {\em Betti numbers of modules of exponent two over regular local rings}, J. Algebra {\bf 193} (1997) 640--659.

\bibitem{SC2} {S.T Chang}, \newblock {\em Betti numbers of modules of essentially monomial type}, Proc. Amer. Math. Soc. {\bf 128} (2000),  1917--1926.

\bibitem{HC} {H. Charalambous}, \newblock {\em Betti numbers of multigraded modules}, J. Algebra {\bf 137} (1991), 491--500.

\bibitem{CU} {M. Crupi, R. Utano},   \newblock {\em Upper bounds for the Betti numbers of graded ideals of a given
length in the exterior algebra}, Comm. Algebra {\bf 27}(9) (1999), 4607--4631.

\bibitem{CU1} {M. Crupi, R. Utano},  \newblock {\em Classes of graded ideals with given data in the exterior algebra}, Comm. Algebra {\bf 35} (2007), 2386--2408.

\bibitem{CF} {M. Crupi, C. Ferr\`{o}}, \newblock {\em Hilbert functions and Betti numbers of reverse lexicographic ideals in the exterior algebra},  \newblock Turk. J. Math. {\bf 36} (2012), 366--375.

\bibitem{CF2} {M. Crupi, C. Ferr\`{o}}, \newblock {\em Revlex ideals in the exterior algebra}, \newblock Math. Reports \textbf{15}(65), 3, (2013), 193--201.

\bibitem{TD}  {T. Deery}, \newblock {\em Revlex segment ideals and minimal Betti numbers}, \newblock Queen's Papers {\bf 129} (1996), 194--219.

\bibitem{DE} {D. Erman}, \newblock {\em A special case of the Buchsbaum-Eisenbud-Horrocks rank conjecture}, Math. Res. Lett. {\bf 17} (2010), 1079--1089.

\bibitem{EG} {E. G. Evans, P. Griffith}, \newblock {\em Binomial behavior of Betti numbers for modules of finite
length}, Pacific J. Math. {\bf 133} (1988), 267--276.

\bibitem{Go}  {M. Goff}, \newblock {\em Bounding Betti numbers of bipartite graph ideals}, \newblock J. Pure Appl. Algebra {\bf 213} (2009), 1170--1172.

\bibitem{MG}  {M. Green}, \newblock {\em Restrictions of linear series in hyperplanes, and some results of Macaulay and Gotzmann}, \newblock Lecture Notes in Math. {\bf 1389}, Springer Berlin (1989), 76--86.

\bibitem{RH} {R. Hartshorne}, \newblock {\em Algebraic vector bundles on projective spaces: a problem list}, \newblock Topology {\bf 18} (1979), 117--128.
     
\bibitem{JT} {J. Herzog, T. Hibi}, {\em Monomial ideals}, Graduate texts in Mathematics, Springer, 2010.

\bibitem{HH}  {H. Hulett}, \newblock {\em Maximum Betti numbers for a given Hilbert function}, \newblock Comm. Algebra {\bf 21} (1993), 2335--2350.

\bibitem{MM} {J. Mermin, S. Murai}, \newblock {\em Lex-Plus-Powers Conjecture holds for pure powers}, \newblock Adv. in Math. {\bf 226} (2011),  3511--3539.

\bibitem{NR} {U. Nagel, V. Reiner}, \newblock {\em Betti Numbers of monomial ideals and shifted skew shapes}, \newblock Electron. J. Comb. {\bf 16}, (2009) $\sharp$ R3, 1--59.

\bibitem{KP} {K. Pardue}, \newblock {\em Deformation classes of graded modules and maximal Betti numbers}, \newblock Illinois J. Math. {\bf 40} (1996), 564--585.

\bibitem{TR}{T. R\"{o}mer}, \newblock {\em Bounds for Betti numbers}, J. Algebra {\bf 249} (2002), 20--37.

\bibitem{Va}  {G. Valla}, \newblock {\em On the Betti numbers of  perfect ideals}, Compositio Math. {\bf 91} (1994) 305--319.

\end{thebibliography}
\end{document}